\documentclass[a4paper]{amsart}
\usepackage[english]{babel}

\usepackage[a4paper,top=3cm,bottom=2cm,left=3cm,right=3cm,marginparwidth=1.75cm]{geometry}

\usepackage{amsmath}
\usepackage{graphicx}
\usepackage[colorinlistoftodos]{todonotes}
\usepackage[colorlinks=true, allcolors=blue]{hyperref}

\usepackage{lmodern}
\usepackage[english]{babel}
\usepackage{amsmath}
\usepackage{color}
\usepackage{amsfonts}
\usepackage{amssymb}
\usepackage{amsthm}
\usepackage{stmaryrd}

\usepackage{pstricks}
\usepackage{epsfig}
\newtheorem{teo}{Theorem}[section]
\newtheorem{pro}[teo]{Proposition}
\newtheorem{lem}[teo]{Lemma}
\newtheorem{cor}[teo]{Corollary}
\newtheorem{rem}[teo]{Remark}

\title{Schur-Weyl duality and the Product of randomly-rotated symmetries by a unitary Brownian motion} 

\author[N. Demni]{Nizar Demni}
\address{IRMAR, Universit\'e de Rennes 1\\ Campus de
Beaulieu\\ 35042 Rennes cedex\\ France}
\email{nizar.demni@univ-rennes1.fr}
\author[T. Hamdi]{Tarek Hamdi}
\address{Department of Information Management System \\ College of Business Administration  \\ Qassim University \\ Buraydah  51452 \\ Saudi Arabia 
and Laboratoire d'Analyse Math\'ematiques et applications LR11ES11 \\ Universit\'e de Tunis El-Manar \\ Tunisie}
\email{t.hamdi@qu.edu.sa}

\begin{document}
\maketitle
\begin{abstract} 
In this paper, we introduce and study a unitary matrix-valued process which is closely related to the Hermitian matrix-Jacobi process. It is precisely defined as the product of a deterministic self-adjoint symmetry and a randomly-rotated one by a unitary Brownian motion. Using stochastic calculus and the action of the symmetric group on tensor powers, we derive an autonomous ordinary differential equation for the moments of its fixed-time marginals. Next, we derive an expression of these moments which involves a unitary bridge between our unitary process and another independent unitary Brownian motion. This bridge motivates and allows to write a second direct proof of the obtained moment expression.
\end{abstract}

\tableofcontents

\section{Introduction}

The Jacobi Unitary Ensemble (JUE) is a unitarily-invariant matrix model which admits various relevant applications including multivariate analysis of variance (\cite{Joh}), statistical physics (\cite{Erd-Far}, \cite{Viv}) and optical-fiber communication (\cite{DFS}, \cite{BDN}). 
Being a multivariate extension of the Beta distribution, it is naturally built out of independent and invertible Wishart matrices. As proved in \cite{Col}, it is also distributed as the product of randomly rotated projections by a Haar unitary matrix. As such,  
the JUE encodes the statistical information about the angles between two uniformly randomly-rotated subspaces. In the large-size limit, these subspaces behave like freely-independent orthogonal projections in some $W^{\star}$-probability space 
$(\mathcal{A}, \tau)$. More generally, given a Haar unitary operator $U$ and two sub-algebras $A, B$ of $\mathcal{A}$ which are freely independent from $\{U,U^{\star}\}$, then $A$ and the rotated sub-algebra $UBU^{\star}$ are freely-independent as well. This is no longer true if we replace $U$ by any fixed-time marginal of a free unitary Brownian motion $(U_t)_{t\geq0}$ (which is freely independent from $A, B$) except in the large-time limit, since $(U_t)_{t\geq0}$ converges strongly to a Haar unitary operator (\cite{CDK}). Nonetheless, when each subalgebra is generated by a single orthogonal projection, one gets an instance of the so-called Voiculescu's liberation process (\cite{Voi}) or in a different guise, the free Jacobi process (\cite{Dem}). In this respect, an extensive spectral study of fixed-time marginals of this process was performed in \cite{Hia-Miy-Ued, Hia-Ued, DHH, Dem-Hmi, CK, Izu-Ued, Demni, Dem-Ham,Tar1,Tar2,Tar3}. In particular, it turns out that the spectral dynamics of the free Jacobi process are governed by that of 
\begin{equation*}
A_t:= RU_tSU_t^{\star}, \quad t \geq 0,
\end{equation*}
where $R, S$ are self-adjoint symmetries which are freely-independent from $\{U,U^{\star}\}$. This unitary process appears naturally in the binomial-type formula proved in \cite{Tar3} and implicitly in \cite{Izu-Ued} where the authors make use of the so-called Geronimus trick. It is also a deformation of the free unitary Brownian motion up to a deterministic time-change. Indeed, if $R = S, \tau(S) = 0$ then Lemma 3.8 in \cite{Haa-Lar} implies that $(SU_tS)_{t \geq 0}$ and $(U_t^{\star})_{t \geq 0}$ are freely-independent which,  together with the multiplicative convolution property of the law of $U_t$, show that $(A_t)_{t \geq 0}$ is distributed as $(U_{2t})_{t \geq 0}$ in $(\mathcal{A}, \tau)$. A different form of this claim states that equality in distribution at fixed time still holds when $\tau(R) = \tau(S) = 0$  
(\cite{Izu-Ued}, \cite{Tar3}). More generally, if $R \neq S$ and $\tau(S) = 0$, then we can write 
\begin{equation*}
A_t = (RS)(SU_tSU_t^{\star}) = (RS)(SU_tS)(U_t^{\star})
\end{equation*}
and infer from Haagerup-Larsen Lemma cited above that $(SU_tS)_{t \geq 0}$ and $(U_t^{\star})_{t \geq 0}$ are freely-independent. We can check further from the very definition of freeness that $(RS)$ and $(SU_tS)_{t \geq 0}$ are so. However, $(RS)$ and $SU_tSU_t^{\star}$ are not freely-independent in general since otherwise, $A_t$ would have the same distribution as $RSU_{2t}$ which is absolutely continuous with respect to the Haar measure on the circle (\cite{Zho}). This is in contradiction with the Lebesgue decomposition of the spectral measure of $A_t$ which admits an atomic part whenever $\tau(R) \neq 0$ (\cite{Tar3}). In a nutshell, the distribution of $RU_tSU_t^{\star}$ is somehow close to that of $U_{2t}$ but, up to our best knowledge, we can not relate both of them using freeness properties except when $\tau(R) = \tau(S) = 0$. 

In order to get more insight into the moments of $A_t$ and in particular to see how close they are to the moments of $U_{2t}$, we shall be interested in the matrix-valued unitary process: 
\begin{equation*}
A_t^N:=R^NU_t^NS^N(U_t^N)^{\star}, \quad t \geq 0,
\end{equation*}
where now $(U_t^N)_{t \geq 0}$ is a $N \times N$ Brownian motion on the unitary group, and $R^N, S^N$ are deterministic matrix-valued self-adjoint symmetries. By analogy with the free setting, this process is connected with the so-called Hermitian matrix-Jacobi process. More precisely, if $\mathbb{E}$ stands for the expectation of the underlying probability space and 
\begin{equation*}
P^N := \frac{I^N+R^N}{2}, \quad Q^N := \frac{I^N+S^N}{2},
\end{equation*}
are the orthogonal projections associated to $R^N, S^N$ ($I^N$ being the $N \times N$ identity matrix), then one has for any $k \geq 1$:
\begin{multline}\label{Binom}
\frac{1}{N}\mathbb{E}\left[\textrm{Tr}(P^NU_t^NQ^N(U_t^N)^{\star}P^N)^k\right] 
=\frac{1}{2^{2k+1}} \binom{2k}{k} + \frac{1}{4N}\textrm{Tr}(R^N+ S^N) +\\ 
\frac{1}{2^{2k}}\sum_{n=1}^k \binom{2k}{k-n}\frac{1}{N}\mathbb{E}\left[\textrm{Tr}((A_t^N)^n)\right],
\end{multline}
where $\textrm{Tr}$ denotes the trace functional. In \cite{DD}, an expression of the LHS of \eqref{Binom} was obtained relying on the semi-group density of the Hermitian Jacobi process, yet it does not allow in its present form to prove a large-N limiting result. 
As a matter of fact, it is challenging to seek more simpler expressions which open the way to compute the moments of the free Jacobi process so far known in few cases. 

On the other hand, combinatorial integration formulas for the expectation of traces of tensor powers of $U_t^N$ were obtained in \cite{Levy} and subsequently in \cite{Dah} where tensor powers of the complex-conjugate of an independent copy of $U_t^N$ are further allowed. However, we can not appeal to these formulas directly to derive expressions for 
\begin{equation}\label{TF}
F_n^N(t):= \frac{1}{N}\mathbb{E}\left[\mathop{\textrm{Tr}}((A_t^N)^n)\right], \quad n \geq 1,
\end{equation}
since in our setting, both $U_t^N$ and its adjoint are gathered in the same trace functional. Nonetheless, the main ideas used in those papers may be adapted here to compute the sequence $F_n^N(t), n \geq 1$. Indeed, the trace functional 
\eqref{TF} may be `linearized by substituting the powers $(A_t^N)^n$ in the space $\mathcal{M}_N$ of $N \times N$ matrices by the $n$-fold tensor power $(A_t^N)^{\otimes n}$ in $\mathcal{M}_N^{\otimes n}$. Moreover, the action of the symmetric group $S_n$ on tensors  leads to: 
\begin{align*}
\mathbb{E}\left[\textrm{Tr}(A_t^N)^n)\right] = \mathbb{E}\left[\textrm{Tr}\left((1\ldots n)(A_t^N)^{\otimes n}\right)\right]= \textrm{Tr}\left((1\ldots n)\mathbb{E}[(A_t^N)^{\otimes n}]\right),
\end{align*}
where we use the same notation for the trace functional in $\mathcal{M}_N^{\otimes n}$. 

Our first main result is the following Theorem where we prove the finite-dimensional analogue of the ordinary differential equation (ODE) satisfied by $RUSU^{\star}$ and proved in \cite[Proposition 2.1]{Tar1}:
\begin{teo}\label{ode}
For any $n\geq 2$,
\begin{align*}
\frac{d}{dt} F_n^N(t)  =&-nF_n^N(t) -n\sum_{ p=1}^{ n-1}\mathbb{E}[\textrm{tr}((A_t^N)^{p})\textrm{tr}((A_t^N)^{n-p})]  +
 \begin{cases}
n^2\alpha^N\beta^N \quad: {\rm n\ odd}
\\ \displaystyle \frac{n^2}{2}((\alpha^N)^2+(\beta^N)^2)\quad: {\rm n\ even}
\end{cases},
\end{align*}
where 
\begin{equation*}
\alpha^N := \frac{1}{N}\textrm{Tr}(R^N), \quad \beta^N:= \frac{1}{N}\textrm{Tr}(S^N). 
\end{equation*}
\end{teo}
As to the stochastic differential equation (SDE) satisfied by $(A_t^N)_{t \geq 0}$ it is not autonomous in contrast with the Hermitian Jacobi process (see e.g. \cite{DD}), and it yields in turn a non autonomous ODE for 
\begin{equation*}
G_n^N(t) := \mathbb{E}[(A_t^N)^{\otimes n}], \quad n \geq 2. 
\end{equation*}

Our second main result is partly independent from the first one and provides an expression of $F_n^N(t)$ which we obtain after solving the ODE satisfied by $(1\dots n)G_n^N$. 

\begin{teo}\label{expression}
Define
\begin{equation*}
\nu_n^N(t):= \frac{1}{N}\mathbb{E}\left[\textrm{Tr}(R^NS^NU_{t}^N)^n\right],
\end{equation*}
and let $(V_t^N)_{t \geq 0}$ be an independent copy of $(U_t^N)_{t \geq 0}$. Then, for any $n\geq 2$,
\begin{align}\label{Exp}
F_n^N(t)=&\nu_n^N(2 t)+ \beta^N\frac{n}{N} \int_0^t\mathbb{E}\left[\textrm{Tr}\left((B_s^N)^{ n-1}V_{2(t-s)}^NR^N\right)\right]ds \nonumber 
\\&+\frac{n}{N^2}\sum_{i=1}^{ n-1}\int_0^t\mathbb{E}\left[\textrm{Tr}\left((B_s^N)^{n-i-1}V_{2(t-s)}^NR^N\right)\textrm{Tr}\left((B_s^N)^{i-1}V_{2(t-s)}^N R^N\right)\right]ds,
\end{align}
where $(B_s^N)_{0 \leq s \leq t}$ denotes the unitary bridge $\left(V_{2(t-s)}^NA_s^N\right)_{0 \leq s \leq t}$.
\end{teo}
By standard arguments on asymptotic freeness, we already know that the LHS of \eqref{Exp} as well as the first and the second terms of its RHS converge in the large-N limit. Consequently, the last term of the RHS converges too. Moreover, the $n$-th moment 
$\nu_n^N(2 t)$ converges as $N \rightarrow \infty$ to the $n$-th moment of $RSU_{2t}$. As a matter of fact, if $\beta^N \rightarrow 0$ then the second term of the RHS of \eqref{Exp} vanishes in the large-N limit while the last one gives the price to compensate for the free-dependence between $RS$ and $SUSU^{\star}$ alluded to before when $\tau(S) = 0$. On the other hand, if $\textrm{Tr}(R^N) = 0$ then the invariance of the distribution of $(U_t^N)_{t \geq 0}$ under the adjoint action of the unitary group shows that 
\begin{equation*}
\mathbb{E}\left[\textrm{Tr}\left((B_s^N)^{ n-1}V_{2(t-s)}^NR^N\right)\right] = 0 
\end{equation*}
and that the integrand of the last term of the RHS of \eqref{Exp} reduces to the covariance of 
\begin{equation*}
\textrm{Tr}\left((B_s^N)^{n-i-1}V_{2(t-s)}^NR^N\right), \quad \textrm{Tr}\left((B_s^N)^{i-1}V_{2(t-s)}^N R^N\right).
\end{equation*}
Finally, the occurrence of the unitary bridge  $(B_s)_{0 \leq s \leq t}$ is intriguing and it would be quite interesting to justify it devoid of analysis. Nonetheless, we shall use this process to write another direct proof of Theorem \ref{expression}.

The paper is organized as follows. In the next section, we recall some facts about the unitary Brownian motion and its relation to the Schur-Weyl duality. In section 3, we perform the stochastic analysis of the tensor power process $((A_t^N)^{\otimes n})_{t \geq 0}$ and deduce the ODE for $(G_n^N(t))_{n \geq 1}$. In the same section, we derive the autonomous ODE for $(F_n^N(t))_{n \geq 1}$ which is of independent interest. In section 4, we prove both Theorems \ref{ode} and \ref{expression}. The last section contains the second proof of Theorem \ref{expression}. 
 
For ease of notations, we shall ommit the dependence on $N$ of the matrices occurring below and hope there will be no confusion with the notations of their free counterparts. We shall also denote $\textrm{tr}$ the normalized trace functional $(1/N)\textrm{Tr}$ acting either on $\mathcal{M}_N$ or on $\mathcal{M}_N^{\otimes n}$.

\section{Reminder: unitary Brownian motion and Schur Weyl duality}
\subsection{Brownian motion in $\mathbb{U}_N$}
Let $\mathbb{U}_N$ be the group of $N \times N$ unitary matrices and $\mathrm{u}_N$ be its Lie algebra of skew-Hermitian matrices in $\mathcal{M}_N$ equipped with the normalized killing form\footnote{This scaling is needed in order to get a non trivial large-N limit, \cite{Levy}.}:
\begin{align*}
\langle A,B \rangle = -N\mathop{\textrm{Tr}}(AB),
\end{align*}
Fix an orthonormal basis $\mathcal{B}$ of $\mathrm{u}_N$. Then, the Brownian motion in $\mathrm{u}_N$ is the skew-Hermitian process 
 $(X_t)_{t \geq 0}$ defined by
 \begin{equation*}
 X_t:=\sum_{\xi\in \mathcal{B}} B_{\xi}(t) \xi,
 \end{equation*}
where $\{B_{\xi}:\xi \in \mathcal{B}\}$ are i.i.d. standard real Brownian motions. This process is independent of the choice of the orthonormal basis and as such, we shall choose in the sequel: 
\begin{align*}
\mathcal{B}=\bigg\{\frac{1}{\sqrt{2N}}(E_{k,l}-E_{l,k}),\frac{i}{\sqrt{2N}}(E_{k,l}+E_{l,k}): 1\leq k<l\leq N\bigg\}\cup\bigg\{\frac{i}{\sqrt{N}}E_{k,k}:1\leq k \leq N\bigg\}.
\end{align*}

The corresponding Brownian motion $(U_t)_{t \geq 0}$ in $\mathbb{U}_N$ is then obtained by wrapping the skew-Hermitian Brownian motion $X_N$ (\cite{Lia}). More concretely, it is the unique strong solution of the following stochastic differential equation 
(hearafter SDE): 
\begin{equation*}
dU_t = U_t dX_t-\frac{1}{2}U_tdt, \quad U_0=I.
\end{equation*}
It is also a left L\'evy process, that is the right increment $U_s^{\star}U_t$ is independent of $(U_s)_{0 \leq s \leq t}$. This choice is by no means a loss of generality since $(U_t^{\star})_{t \geq 0}$ is a right L\'evy process and has the same distribution as 
$(U_t^{\star})_{t \geq 0}$.

\subsection{Schur-Weyl duality}
Let $V$ be a vector space of dimension $\geq 2$. Then, the symmetric group $S_n$ acts on the tensor power $V^{\otimes n}$ by permuting its factors, namely:
\begin{align*}
[\sigma](v_1\otimes\ldots \otimes v_n)=v_{\sigma^{-1}(1)}\otimes\ldots \otimes v_{\sigma^{-1}(n)}.
\end{align*}
This gives rise to a  representation which is `dual' to the standard representation of the linear group $GL(V)$ on $V^{\otimes n}$:
\begin{align*}
g(v_1\otimes\ldots \otimes v_n)=gv_1\otimes\ldots \otimes gv_n,
\end{align*}
in the sense that these two actions commute and are full mutual centralizers in the algebra End$(V^{\otimes n})$. The last statement, known as the Schur-Weyl duality, plays a key role  in T. L\'evy's approach to the heat kernel on $\mathbb{U}_N$ (\cite{Levy}). 
Take $V=\mathbb{C}^N$ so that $\mathcal{M}_N = \textrm{End}(V)$. Then, for any permutation $\sigma \in S_n$ with cycle decomposition 
\begin{align*}
\sigma=(i_1^1\ldots i_{l_1}^1)\ldots(i_1^r\ldots i_{l_r}^r),
\end{align*}
and for any collection $M_1,\ldots M_n\in\mathcal{M}_N$, we have
\begin{align}\label{trace}
\textrm{Tr}\big([\sigma](M_1\otimes\ldots \otimes M_n)  \big)= \textrm{Tr} \big(M_{i_{l_1}^1}\ldots  M_{i_1^1}  \big)\ldots \mathop{\textrm{Tr}}\big(M_{i_{l_r}^r}\ldots  M_{i_1^r}  \big).
\end{align}
In particular, for any one-cycle:
\begin{align*}
\textrm{Tr}\big([(i_1\ldots i_r)](M_1\otimes \ldots \otimes M_n)  \big)=\textrm{Tr}\big(M_{i_r}\ldots  M_{i_1}  \big), 
\end{align*}
and in turn
\begin{align}\label{mom-BM}
\mu_n^N(t) := \mathbb{E}\left[\mathop{\textrm{tr}}(U_{t}^N)^n\right] = \mathbb{E}\left[\mathop{\textrm{tr}}\left((1\ldots n)(U_t^N)^{\otimes n}\right)\right]. 
\end{align}
Moreover, the finite-variation part of the semi-martingale $(U_t)^{\otimes n}$ is given by (\cite{Dah}):
\begin{equation}\label{UBM}
-(U_t)^{\otimes n} \left[\frac{n}{2} + \frac{1}{N}\sum_{1\leq i<j\leq n} [(ij)]\right]dt,
\end{equation}
whence
\begin{align}\label{BM}
\mathbb{E}[(U_t)^{\otimes n}]=e^{-nt/2}\exp\left(-\frac{t}{N}\sum_{1\leq i<j\leq n}[(ij)] \right).
\end{align}
On the other hand, the sequence $\nu_n(t), N \geq n$ was computed explicitly in \cite{Biane} when the traces of $R$ and $S$ vanish simultaneously:
\begin{align*}
\mu_n^N(t)=\frac{e^{-nt/2}}{N}\sum_{k=0}^{n-1}(-1)^k\binom{N+n-1-k}{n}\binom{n-1}{k}e^{-t(n^2-(2k+1)n)/(2N)}.
\end{align*}



\section{Stochastic analysis of $A^{\otimes n}$ and the ODE satisfied by $(G_n)_{n \geq 1}$}
Let $R,S\in \mathcal{M}_N$ be two self-adjoint symmetries and recall the notations $\beta=\mathop{\textrm{tr}}(S), \xi=\mathop{\textrm{tr}}(RS)$. 
We start with the following elementary lemma: 
\begin{lem}\label{firstrank}
The unitary process $(A_t)_{t \geq 0}$ satisfies: 
\begin{align*}
dA_t = R U_t (dX_tS-SdX_t)U_t^{\star}+(\beta R-A_t)dt, \quad A_0=RS.
\end{align*}
\end{lem}
\begin{proof}
From It\^o's formula, we have:
\begin{equation*}
dA_t = R[dU_t]SU_t^{\star}+RU_tS[dU_t^{\star}]+R(dU_t)S(dU_t^{\star}),
\end{equation*}
where $R(dU_t)S(dU_t^{\star})$ is the bracket of the semimartingales $RdU_t$ and $SdU_t^{\star}$. Since
\begin{equation*}
dU_t^{\star} = -dX_t^NU_t^{\star}-\frac{1}{2}U_t^{\star}dt,
\end{equation*}
and $R(dU_t)S(dU_t^{\star}) = \beta R$ which is readily checked componentwise, the lemma is proved. 
\end{proof}

Next, we proceed to the stochastic analysis of the $n$-fold tensor power $A^{\otimes n}$. To this end, we introduce the following notations: for any $N \times N$ matrices $M,D$,
\begin{eqnarray*}
(M)_i &= & I^{\otimes i-1}\otimes M\otimes I^{\otimes n-i}, \\
(M\otimes D)_{i,j} &=& I^{\otimes i-1}\otimes M\otimes I^{\otimes j-i-1}\otimes D\otimes I^{\otimes n-j}.
\end{eqnarray*}
Set also $dZ_t := RU_t(dX_tS-SdX_t)U_t^{\star}$ so that $dA_t=dZ_t+(\beta R-A_t)dt$. Then, 
\begin{pro}\label{SDE}
For any $n\geq1$ and any $t\geq0$, one has
\begin{align*}
dA_t^{\otimes n}=&\sum_{i=1}^nA_t^{\otimes i-1}\otimes dZ_t\otimes A_t^{\otimes n-i}+\beta\sum_{i=1}^nA_t^{\otimes n}(U_tSU_t^{\star})_idt
-A_t^{\otimes n}\left[n+\frac{2}{N}\sum_{1\leq i<j\leq n}[(ij)]\right]dt
\\&+\frac{1}{N}\sum_{\substack{
     1\leq i<j\leq n\\
      1\leq k,l\leq N
}} (A_t)^{\otimes n}(U_tSE_{k,l}SU_t^{\star}\otimes U_tE_{l,k}U_t^{\star})_{i,j}dt
\\&+\frac{1}{N}\sum_{\substack{
     1\leq i<j\leq n\\
      1\leq k,l\leq N
}} (A_t)^{\otimes n}(U_tE_{k,l}U_t^{\star}\otimes U_tSE_{l,k}SU_t^{\star})_{i,j}dt.
\end{align*}
\end{pro}

\begin{proof}
Applied to the tensor power $A^{\otimes n}$, the It\^o's formula reads:
\begin{align}
dA_t^{\otimes n} &=\sum_{i=1}^nA_t^{\otimes i-1}\otimes dA_t\otimes A_t^{\otimes n-i} +\sum_{1\leq i<j\leq n}A_t^{\otimes i-1}\otimes dA_t\otimes A_t^{\otimes j-i-1}\otimes dA_t\otimes A_t^{\otimes n-j} \nonumber 
\\& = \sum_{i=1}^nA_t^{\otimes i-1}\otimes dA_t\otimes A_t^{\otimes n-i} +\sum_{1\leq i<j\leq n}A_t^{\otimes i-1}\otimes dZ_t \otimes A_t^{\otimes j-i-1}\otimes dZ_t \otimes A_t^{\otimes n-j} \label{ito}.
\end{align}
From Lemma \ref{firstrank}, the first sum becomes:
\begin{align*}
\sum_{i=1}^nA_t^{\otimes i-1}\otimes dZ_t\otimes A_t^{\otimes n-i}+\beta \sum_{i=1}^nA_t^{\otimes i-1}\otimes R\otimes A_t^{\otimes n-i}dt-nA_t^{\otimes n}dt.
\end{align*}
But $A_t$ is a unitary matrix, therefore:
\begin{align*}
A_t^{\otimes i-1}\otimes R\otimes A_t^{\otimes n-i}=A_t^{\otimes n}(A_t^{\star}R)_i=A_t^{\otimes n}(U_tSU_t^{\star})_i,
\end{align*}
whence we conclude that the first sum in \eqref{ito} may be written as:
\begin{align*}
\sum_{i=1}^nA_t^{\otimes i-1}\otimes dZ_t\otimes A_t^{\otimes n-i}+\beta \sum_{i=1}^nA_t^{\otimes n}(U_tSU_t^{\star})_idt-nA_t^{\otimes n}dt.
\end{align*}
Now, consider the $(i,j)$-th term in the second sum in \eqref{ito}: 
\begin{align*}
A_t^{\otimes i-1}\otimes dZ_t\otimes A_t^{\otimes j-i-1}\otimes dZ_t\otimes A_t^{\otimes n-j}.
\end{align*}
From the very definition of $dZ_t$, this term splits into four terms:
\begin{align}
& A_t^{\otimes i-1}\otimes RU_tdX_tSU_t^{\star}\otimes A_t^{\otimes j-i-1}\otimes RU_tdX_tSU_t^{\star}\otimes A_t^{\otimes n-j}\label{s1}
\\& A_t^{\otimes i-1}\otimes RU_tSdX_tU_t^{\star}\otimes A_t^{\otimes j-i-1}\otimes RU_tSdX_tU_t^{\star}\otimes A_t^{\otimes n-j}\label{s2}
\\&- A_t^{\otimes i-1}\otimes RU_tdX_tSU_t^{\star}\otimes A_t^{\otimes j-i-1}\otimes RU_tSdX_tU_t^{\star}\otimes A_t^{\otimes n-j}\label{s3}
\\&- A_t^{\otimes i-1}\otimes RU_tSdX_tU_t^{\star}\otimes A_t^{\otimes j-i-1}\otimes RU_tdX_tSU_t^{\star}\otimes A_t^{\otimes n-j}\label{s4}.
\end{align}
The sum $\eqref{s1} + \eqref{s2}$ may be written as: 
\begin{align*}
&(RU_t)^{\otimes n}I^{\otimes i-1}\otimes dX_t\otimes I^{\otimes j-i-1}\otimes dX_t\otimes I^{\otimes n-j}(SU_t^{\star})^{\otimes n}, 
\\&(RU_tS)^{\otimes n}I^{\otimes i-1}\otimes dX_t\otimes I^{\otimes j-i-1}\otimes dX_t\otimes I^{\otimes n-j}(U_t^{\star})^{\otimes n}.
\end{align*}
Using the decomposition of $X$ in the basis $\mathcal{B}$ and since the bracket of two independent Brownian motions vanish, we get:
\begin{multline*}
\eqref{s1} + \eqref{s2} = -\frac{1}{N}(RU_t)^{\otimes n} \left\{\sum_{k,l=1}^N I^{\otimes i-1}\otimes E_{k,l}\otimes I^{\otimes j-i-1}\otimes E_{l,k}\otimes I^{\otimes n-j}\right\}(SU_t^{\star})^{\otimes n} \\ 
-\frac{1}{N}(RU_tS)^{\otimes n}\left\{\sum_{k,l=1}^N I^{\otimes i-1}\otimes E_{k,l}\otimes I^{\otimes j-i-1}\otimes E_{l,k}\otimes I^{\otimes n-j}\right\}(U_t^{\star})^{\otimes n}.
\end{multline*}
But the terms between brackets act on tensors as the transposition $[(ij)]$ whence  
\begin{align*}
\eqref{s1} + \eqref{s2} = -\frac{1}{N}(RU_t)^{\otimes n}[(ij)](SU_t^{\star})^{\otimes n}dt
-\frac{1}{N}(RU_tS)^{\otimes n}[(ij)](U_t^{\star})^{\otimes n}dt.
\end{align*}
Since the Schur-Weyl representation of any permutation commutes with any tensor power $M^{\otimes n}$, we end up with:
\begin{equation*}
\eqref{s1} + \eqref{s2} = -\frac{2}{N}A_t^{\otimes n}[(ij)]dt.
\end{equation*}
Finally, factoring out $A_t^{\otimes n}$ from \eqref{s3} and \eqref{s4} and decomposing again $X$ as a sum of independent real standard Brownian motions, the same computations lead to: 
\begin{equation}\label{two sums}
   \frac{1}{N} \sum_{1\leq k,l\leq N}A_t^{\otimes n}(U_tSE_{k,l}SU_t^{\star}\otimes U_tE_{l,k}U_t^{\star})_{i,j}+\frac{1}{N} \sum_{1\leq k,l\leq N}A_t^{\otimes n}(U_tE_{k,l}U_t^{\star}\otimes U_tSE_{l,k}SU_t^{\star})_{i,j},
\end{equation}
which prove the proposition.
\end{proof}
\begin{rem}
If the symmetry $S$ is diagonal, then we can readily check that both sums displayed in \eqref{two sums} coincide, therefore $dA_t^{\otimes n}$ reduces to:
\begin{align*}
dA_t^{\otimes n}=&\sum_{i=1}^nA_t^{\otimes i-1}\otimes dZ_t\otimes A_t^{\otimes n-i}+\beta\sum_{i=1}^nA_t^{\otimes n}(U_tSU_t^{\star})_idt
-A_t^{\otimes n}\left[n+\frac{2}{N}\sum_{1\leq i<j\leq n}[(ij)]\right]dt
\\&+\frac{2}{N}\sum_{\substack{
     1\leq i<j\leq n\\
      1\leq k,l\leq N
}} (A_t)^{\otimes n}(U_tSE_{k,l}SU_t^{\star}\otimes U_tE_{l,k}U_t^{\star})_{i,j}dt.
\end{align*}
\end{rem}

Taking the expectation in both sides of the SDE derived in proposition \ref{SDE}, we get the following matrix-valued ODE: 
\begin{cor}\label{tensor_ode}
For any $n\geq 2$,
\begin{align*}
   \frac{d}{dt} G_n(t)=&-G_n(t)\left[n+\frac{2}{N}\sum_{1\leq i<j\leq n}[(ij)]\right]+\beta \mathbb{E}[A_t^{\otimes n}\sum_{i=1}^n(U_tSU_t^{\star})_i]
\\&+\frac{1}{N}\mathbb{E}[A_t^{\otimes n}\sum_{\substack{
     1\leq i<j\leq n\\
      1\leq k,l\leq N
}}(U_tSE_{k,l}SU_t^{\star}\otimes U_tE_{l,k}U_t^{\star})_{i,j}]
\\&+\frac{1}{N}\mathbb{E}[A_t^{\otimes n}\sum_{\substack{
     1\leq i<j\leq n\\
      1\leq k,l\leq N
}}(U_tE_{k,l}U_t^{\star}\otimes U_tSE_{l,k}SU_t^{\star})_{i,j}].
\end{align*}
\end{cor}

Using \eqref{BM}, the solution to this ODE may be written as:
\begin{pro}
For all $n\geq 2$, we have
\begin{align*}
G_n(t)=& [RS]^{\otimes n}\mathbb{E}[(U_{2t})^{\otimes n}]
    +\beta \sum_{i=1}^n\int_0^t \mathbb{E}[(V_{2(t-s)}A_s)^{\otimes n}(U_sSU_s^{\star})_i]ds
    \\&+\frac{1}{N}\sum_{1\leq i<j\leq n}\sum_{1\leq k,l\leq N}\int_0^t \mathbb{E}[(V_{2(t-s)}A_s)^{\otimes n}(U_sSE_{k,l}SU_s^{\star}\otimes U_sE_{l,k}U_s^{\star})_{i,j}]ds
\\&+\frac{1}{N}\sum_{1\leq i<j\leq n}\sum_{1\leq k,l\leq N}\int_0^t \mathbb{E}[(V_{2(t-s)}A_s)^{\otimes n}(U_sE_{k,l}U_s^{\star}\otimes U_sSE_{l,k}SU_s^{\star})_{i,j}]ds,
\end{align*}
where $(V_t)_{t \geq 0}$ is an independent copy of $(U_t)_{t \geq 0}$. 
\end{pro}
\begin{proof}
Since $G_n$ start at $(RS)^{\otimes n}$, then the solution to the ODE in Lemma  \ref{tensor_ode} is
\begin{align*}
    G_n(t)=&[RS]^{\otimes n} \mathbb{E}[(U_{2t})^{\otimes n}] 
    +\beta \sum_{i=1}^n\int_0^t \mathbb{E}[(U_{2s}^{\star}U_{2t})^{\otimes n}]\mathbb{E}[A_s^{\otimes n}(U_sSU_s^{\star})_i]ds
    \\&+\frac{1}{N}\sum_{1\leq i<j\leq n}\sum_{1\leq k,l\leq N}\int_0^t \mathbb{E}[(U_{2s}^{\star}U_{2t})^{\otimes n}]\mathbb{E}[A_s^{\otimes n}(U_sSE_{k,l}SU_s^{\star}\otimes U_sE_{l,k}U_s^{\star})_{i,j}]ds
\\&+\frac{1}{N}\sum_{1\leq i<j\leq n}\sum_{1\leq k,l\leq N}\int_0^t \mathbb{E}[((U_{2s}^{\star}U_{2t})^{\otimes n})^{\otimes n}]\mathbb{E}[A_s^{\otimes n}(U_sE_{k,l}U_s^{\star}\otimes U_sSE_{l,k}SU_s^{\star})_{i,j}]ds.
\end{align*}
Since $(U_t)_{t \geq 0}$ is a left L\'evy process, then for any $0 \leq s \leq t$, the right increment $(U_{2s}^{\star}U_{2t})^{\otimes n}$ is independent from both $\{U_s, U_s^{\star} = U_s^{-1}\}$. Hence, 
\begin{align*}
    G_n(t) & = [RS]^{\otimes n} \mathbb{E}[(U_{2t})^{\otimes n}] +\beta \sum_{i=1}^n\int_0^t \mathbb{E}[(U_{2s}^{\star}U_{2t}A_x)^{\otimes n}(U_sSU_s^{\star})_i]ds
    \\&+\frac{1}{N}\sum_{1\leq i<j\leq n}\sum_{1\leq k,l\leq N}\int_0^t \mathbb{E}[(U_{2s}^{\star}U_{2t}A_s)^{\otimes n}(U_sSE_{k,l}SU_s^{\star}\otimes U_sE_{l,k}U_s^{\star})_{i,j}]ds
\\&+\frac{1}{N}\sum_{1\leq i<j\leq n}\sum_{1\leq k,l\leq N}\int_0^t \mathbb{E}[(U_{2s}^{\star}U_{2t}A_s)^{\otimes n}(U_sE_{k,l}U_s^{\star}\otimes U_sSE_{l,k}SU_s^{\star})_{i,j}]ds.
\end{align*}
Substituting $U_{2s}^{\star}U_{2t}$ by $V_{2s}^{\star}V_{2t}$ and using the stationarity of the right increments of $(V_t)_{t \geq 0}$, we are done.
\end{proof}

\section{Proofs of the main results}
This section is devoted to the proofs of both Theorems \ref{ode} and \ref{expression}. We start with: 
\begin{proof}[Proof of Theorem \ref{ode}]
Since $[(ij)]$ commute with $G_n(t)$ and since $(1\dots n)(ij)$ splits into the two disjoint cycles: 
\begin{equation*}
(12\dots i(j+1)\dots n), \quad ((i+1)\dots j),
\end{equation*}
we readily get:
\begin{align*}
\frac{d}{dt}  F_n(t)  =&\mathbb{E}\textrm{tr}\left([(1\ldots n)]G_n'(t)\right)
  \\=&-nF_n(t)-2\sum_{1\leq i<j \leq n}\mathbb{E}[\textrm{tr}(A_t^{ j-i})\textrm{tr}( A_t^{ n-(j-i)} )] +\beta \sum_{i=1}^n\mathbb{E}\textrm{tr}(A_t^{n-1}R)
  \\&+\frac{1}{N}\sum_{1\leq k,l\leq N}\sum_{1\leq i<j\leq n}\left\{\mathbb{E}\textrm{tr}[A_t^{ n-j+i-1} RU_tE_{k,l}SU_t^{\star} A_t^{ j-i-1} RU_tSE_{l,k}U_t^{\star}]\right.
 \\&\left.+\mathbb{E}\textrm{tr}[A_t^{ n-j+i-1} RU_tSE_{k,l}U_t^{\star} A_t^{ j-i-1} RU_tE_{l,k}SU_t^{\star}]\right\}.
  \end{align*}
Using the following identity, 
\begin{align}\label{magic}
\sum_{1\leq k,l\leq N} E_{k,l}ME_{l,k}= \textrm{Tr}(M)I_n, \quad M \in \mathcal{M}_N, 
\end{align}
it follows that 
\begin{eqnarray*}
\sum_{1\leq k,l\leq N} \textrm{tr}[A_t^{ n-j+i-1} RU_tE_{k,l}SU_t^{\star} A_t^{ j-i-1} RU_tSE_{l,k}U_t^{\star}] & = & \textrm{tr}(A_t^{ n-j+i-1} RU_tU_t^{\star})\textrm{Tr}(SU_t^{\star} A_t^{ j-i-1} RU_tS) \\ 
\sum_{1\leq k,l\leq N}  \textrm{tr}[A_t^{ n-j+i-1} RU_tSE_{k,l}U_t^{\star} A_t^{ j-i-1} RU_tE_{l,k}SU_t^{\star}] &= & \textrm{tr}(A_t^{ n-j+i-1} RU_tSSU_t^{\star})\textrm{Tr}(U_t^{\star} A_t^{ j-i-1} RU_t).
  \end{eqnarray*}
Consequently, 
 \begin{align*}
\frac{d}{dt} F_n(t)  & = -nF_n(t)+\beta \sum_{i=1}^n\mathbb{E}\textrm{tr}(A_t^{n-1}R) -2\sum_{1\leq i<j \leq n}\mathbb{E}[\textrm{tr}(A_t^{ j-i})\textrm{tr}( A_t^{ n-(j-i)} )] 
 \\& + 2\textrm{E}[\textrm{tr}(A_t^{n-j+i-1} R)\textrm{tr}( A_t^{ j-i-1} R)]
 \\& =-nF_n(t)+\beta \sum_{i=1}^n\mathbb{E}\textrm{tr}(A_t^{n-1}R)-2\sum_{ p=1}^{ n-1}(n-p)\mathbb{E}[\textrm{tr}(A_t^{ p})\textrm{tr}( A_t^{ n-p} )] 
  \\&+ 2\sum_{p=1}^{ n-1}(n-p)\mathbb{E}[\textrm{tr}(A_t^{ n-p-1} R)\textrm{tr}( A_t^{ p-1} R)].
  \end{align*}
  But
  \begin{align*}
  2\sum_{ p=1}^{ n-1}(n-p)\mathbb{E}[\textrm{tr}(A_t^{ p})\textrm{tr}( A_t^{ n-p} )] &=\sum_{ p=1}^{ n-1}(n-p)\mathbb{E}[\textrm{tr}(A_t^{ p})\textrm{tr}( A_t^{ n-p} )] +\sum_{p=1}^{n-1}p\mathbb{E}[\textrm{tr}(A_t^{ p})\textrm{tr}( A_t^{ n-p} )] 
 \\=&n\sum_{ p=1}^{n-1}\mathbb{E}[\textrm{tr}(A_t^{ p})\textrm{tr}( A_t^{ n-p} )] ,
  \end{align*}
 and similarly,
  \begin{align*}
  2\sum_{p=1}^{ n-1}(n-p)\mathbb{E}[\textrm{tr}(A_t^{ n-p-1} R)\textrm{tr}( A_t^{ p-1} R)] =n\sum_{p=1}^{ n-1}\mathbb{E}[\textrm{tr}(A_t^{ n-p-1} R)\textrm{tr}( A_t^{ p-1} R)].
  \end{align*}
 Altogether, we get
  \begin{align*}
\frac{d}{dt}  F_n(t)=&-nF_n(t) +\beta \sum_{i=1}^n\mathbb{E}\textrm{tr}(A_t^{n-1}R)-n\sum_{ p=1}^{ n-1}\mathbb{E}[\textrm{tr}(A_t^{ p})\textrm{tr}( A_t^{ n-p} )] 
  \\&+n\sum_{p=1}^{ n-1}\mathbb{E}[\textrm{tr}(A_t^{ n-p-1} R)\textrm{tr}( A_t^{ p-1} R)].
\end{align*}
Finally, since $R^2 = I$ then for any $n\geq1$,
\begin{align*}
 \textrm{tr}(A_t^{n-1}R)= \textrm{tr}(R(RU_tSU_t^{\star})^{n-1}) = 
 \begin{cases}
\textrm{tr}(R)=\alpha \quad: {\rm n\ odd}
\\ \textrm{tr}(S)=\beta \quad: {\rm n\ even}
\end{cases},
\end{align*}
whence it follows that for any $1\leq p\leq n-1, n \geq 2$,
\begin{align*}
  \textrm{tr}(A_t^{ n-p-1} R) \textrm{tr}( A_t^{ p-1} R)&=
  \begin{cases}
  \alpha^2 \quad: {\rm n\ even,\ p\ odd}\\
  \beta^2 \quad: {\rm n\ even,\ p\ even}\\
 \alpha\beta \quad: {\rm n\ odd}
  \end{cases}.
\end{align*}
As a result
\begin{align*}
\frac{d}{dt} F_n(t)  =&-nF_n(t) -n\sum_{ p=1}^{ n-1}\mathbb{E}[\textrm{tr}(A_t^{ p})\textrm{tr}( A_t^{ n-p} )]  +
 \begin{cases}
\alpha\beta(n+n(n-1)) \quad: {\rm n\ odd}
\\ n\beta^2+ \displaystyle \frac{n^2}{2}\alpha^2+\frac{n(n-2)}{2}\beta^2\quad: {\rm n\ even}
\end{cases}.
\end{align*}
The Theorem is proved.
\end{proof}

\begin{rem}
For $n=1$,  $F_1(t)=\mathbb{E}\textrm{tr}(A_t)$ and Lemma \ref{firstrank} yields: 
\begin{align*}
  \frac{d}{dt}  F_1(t)=-F_1(t)+\alpha\beta.
\end{align*}
Hence,
\begin{equation*}
F_1(t)=e^{-t}(\xi -\alpha \beta)+ \alpha \beta . 
\end{equation*}
\end{rem}

We now proceed to the proof of Theorem \ref{expression}. To this end, we recall the process
\begin{align*}
B_s=V_{2(t-s)}A_s=V_{2(t-s)}RU_sSU_s^{\star}, \quad 0 \leq s \leq t,
\end{align*}
which bridges between $V_{2t}$ and $A_t$.

\begin{proof}[Proof of Theorem \ref{expression}]
Letting the one-cycle $(1 \dots n)$ act on the both sides of the formula proved in the last proposition, we readily get: 
\begin{align*}
    F_n(t) & =\textrm{tr}\left([(1\ldots n)] G_n(t)\right)
    \\=&\nu_n(2 t)+\beta \sum_{i=1}^n \int_0^t \mathbb{E}\textrm{tr}\left([(1\ldots n)]  B_s^{\otimes n}(U_sSU_s^{\star})_i\right)ds
    \\&+\frac{1}{N}\sum_{1\leq i<j\leq n}\sum_{1\leq k,l\leq N}\int_0^t \left\{\mathbb{E}\textrm{tr}\left([(1\ldots n)]  B_s^{\otimes n}(U_sSE_{k,l}SU_s^{\star}\otimes U_sE_{l,k}U_s^{\star})_{i,j}\right)\right.
\\&\left.+ \mathbb{E}\textrm{tr}\left([(1\ldots n)] B_s^{\otimes n}(U_sE_{k,l}U_s^{\star}\otimes U_sSE_{l,k}SU_s^{\star})_{i,j}\right)\right\}ds
\\& = \nu_n(2t)+\beta \sum_{i=1}^n \int_0^t \mathbb{E}\textrm{tr}\left(B_s^{n-1}V_{2(t-s)}R\right)ds
    \\&+\frac{1}{N}\sum_{1\leq i<j\leq n}\sum_{1\leq k,l\leq N}\int_0^t \left\{\mathbb{E}\textrm{tr}\left(B_s^{n-j+i-1}V_{2(t-s)}RU_sE_{k,l}SU_s^{\star}B_s^{j-i-1} V_{2(t-s)}RU_sSE_{l,k}U_s^{\star}\right)\right.
\\&\left.+ \mathbb{E}\textrm{tr}\left(B_s^{n-j+i-1}V_{2(t-s)}RU_sSE_{k,l}U_s^{\star}B_s^{j-i-1} V_{2(t-s)}RU_sE_{l,k}SU_s^{\star}\right)\right\}ds.
\end{align*}
Applying the identity \eqref{magic}, we end up with: 
\begin{align*}
F_n(t)=&\nu_n(2t)+n\beta \int_0^t\mathbb{E}\textrm{tr}(B_s^{ n-1}V_{2(t-s)}R)ds
\\&+ 2 \sum_{ i=1}^{n-1}\sum_{j = i+1}^{n}\int_0^t\mathbb{E}[\textrm{tr}(B_s^{ n-j+i-1}V_{2(t-s)}R)\textrm{tr}(B_s^{j-i-1}V_{2(t-s)} R)]ds 
\\& = \nu_n(2t)+n\beta \int_0^t\mathbb{E}\textrm{tr}(B_s^{ n-1}V_{2(t-s)}R)ds
\\&+ 2 \sum_{p=1}^{n-1}(n-p)\int_0^t\mathbb{E}[\textrm{tr}(B_s^{ n-p-1}V_{2(t-s)}R)\textrm{tr}(B_s^{p-1}V_{2(t-s)} R)]ds
\\& = \nu_n(2t)+n\beta \int_0^t\mathbb{E}\textrm{tr}(B_s^{ n-1}V_{2(t-s)}R)ds
\\&+ n\sum_{p=1}^{n-1}\int_0^t\mathbb{E}[\textrm{tr}(B_s^{n-p-1}V_{2(t-s)}R)\textrm{tr}(B_s^{p-1}V_{2(t-s)} R)]ds.
\end{align*}
\end{proof}

\begin{cor}\label{trzero}
Assume $\textrm{Tr}(R)=0$. Then, for any $n \geq 2$,
\begin{equation*}
 F_n(t)=\nu_n(2 t)+n\sum_{ i=1}^{n-1}\int_0^t{\rm Cov}[\textrm{tr}(B_s^{ n-i-1}V_{2(t-s)}R),\textrm{tr}(B_s^{i-1}V_{2(t-s)} R)]ds .
\end{equation*}
\end{cor}
\begin{proof}
Since $\textrm{Tr}(R)=0$ then $N$ is even and we can find a rotation $\rho$ such that $\rho R \rho^{\star}=-R$. But the law of $B_s^{n-1}V_{2(t-s)}$ is invariant under rotations therefore for any $j \geq 1$:
\begin{equation*}
    \mathbb{E}\textrm{tr}(B_s^{j-1}V_{2(t-s)} R)=\mathbb{E}\textrm{tr}(\rho^*B_s^{j-1}V_{2(t-s)} \rho R)=-\mathbb{E}\textrm{tr}(B_s^{j-1}V_{2(t-s)} R).
\end{equation*}
Consequently,
\begin{equation*}
    \mathbb{E}\textrm{tr}(B_s^{n-1}V_{2(t-s)} R)=0,
\end{equation*}
and
\begin{equation*}
  \mathbb{E}[\textrm{tr}(B_s^{ n-i-1}V_{2(t-s)}R)\textrm{tr}(B_s^{i-1}V_{2(t-s)} R)]={\rm Cov}[\textrm{tr}(B_s^{ n-i-1}V_{2(t-s)}R),\textrm{tr}(B_s^{i-1}V_{2(t-s)} R)],
\end{equation*}
as claimed.
\end{proof}
\begin{rem}
The same expression obviously holds when $\beta = 0$. This is in agreement with the fact that $R$ and $S$ play the same role since $U$ and $U^{\star}$ have the same distribution. 
\end{rem}

\section{A second proof of Theorem \ref{expression}}
In this section, we write a second proof of Theorem \ref{expression} relying on the bridge
\begin{align*}
B_s=V_{2(t-s)}A_s, \quad s\in[0,t].
\end{align*} 
To this end, we define 
\begin{align*}
H_n(s) := \mathbb{E}[\textrm{tr} (B_s)^n]=\mathbb{E}[\textrm{tr}(V_{2(t-s)}A_s)^n]
\end{align*}
so that we obviously have $\mathbb{E}[\textrm{tr} B_0^n]=\mathbb{E}[\textrm{tr}(V_{2t}RS)^n]=\nu_n(2 t)$ and $\mathbb{E}[\textrm{tr}B_t^n]=\mathbb{E}[\textrm{tr}A_t^n]=F_n(t)$. 
\begin{proof}[Second proof of Theorem \ref{expression}]
Consider the matrix-valued function: 
\begin{align*}
K_n(s)=\mathbb{E}\left(( B_s)^{\otimes n}\right)=\mathbb{E}\left((V_{2(t-s)})^{\otimes n}\right)\mathbb{E}\left((A_s)^{\otimes n}\right)=\tilde G_n(s)G_n(s),
\end{align*}
where $\tilde G_n(s)=\mathbb{E}\left((V_{2(t-s)})^{\otimes n}\right)$. 
Using 
\begin{align}\label{homogene}
   \frac{d}{ds} \tilde G_n(s)=n\tilde G_n(s)+\frac{2}{N}\sum_{1\leq i<j\leq n}\tilde G_n(s)[(ij)],
\end{align}
which readily follows from \eqref{BM}, together with Corollary \ref{tensor_ode} and the independence of $V$ and $A$, we derive: 
\begin{align*}
 \frac{d}{ds}K_n(s) & = \frac{d\tilde G_n}{ds} (s)G_n(s) + \tilde G_n(s)\frac{dG_n}{ds}(s)
 = \beta \mathbb{E}\left[V_{2(t-s)}^{\otimes n}A_s^{\otimes n}\sum_{i=1}^n(U_sSU_s^{\star})_i\right]
    \\&+\frac{1}{N}\mathbb{E}\left[V_{2(t-s)}^{\otimes n}A_s^{\otimes n}\sum_{1\leq i<j\leq n}\sum_{1\leq k,l\leq N}(U_sSE_{k,l}SU_s^{\star}\otimes U_sE_{l,k}U_s^{\star})_{i,j}\right]
\\&+\frac{1}{N}\mathbb{E}\left[V_{2(t-s)}^{\otimes n}A_s^{\otimes n}\sum_{1\leq i<j\leq n}\sum_{1\leq k,l\leq N}(U_sE_{k,l}U_s^{\star}\otimes U_sSE_{l,k}SU_s^{\star})_{i,j}\right]
\\=&\beta \mathbb{E}\left[B_s^{\otimes n}\sum_{i=1}^n(U_sSU_s^*)_i\right]
\\&+\frac{1}{N}\mathbb{E}\left[B_s^{\otimes n}\sum_{1\leq i<j\leq n}\sum_{1\leq k,l\leq N}(U_sSE_{k,l}SU_s^{\star}\otimes U_sE_{l,k}U_s^{\star})_{i,j}+(U_sE_{k,l}U_s^{\star}\otimes U_sSE_{l,k}SU_s^{\star})_{i,j}\right]
\end{align*}
Applying $\textrm{tr}(1\ldots n)$ to both sides of the last equality, we further get:
\begin{align*}
\frac{d}{ds} H_n(s) &= \textrm{tr}\left[(1\ldots n)\frac{dK_n}{ds}(s)\right] =n\beta \mathbb{E}\textrm{tr}[ B_s^{ n}U_sSU_s^{\star}]
\\&+\frac{1}{N}\sum_{1\leq i<j\leq n}\sum_{1\leq k,l\leq N}\left\{\mathbb{E}\textrm{tr}\left[B_s^{ n-j+i}U_sSE_{k,l}SU_s^{\star} B_s^{j-i} U_sE_{l,k}U_s^{\star}\right]\right.
\\&\left. +\mathbb{E}\textrm{tr}\left[ B_s^{ n-j+i}U_sE_{k,l}U_s^{\star} B_s^{j-i} U_sSE_{l,k}SU_s^{\star}\right]\right\}.
\end{align*}
Performing the index change $p=j-i$ in the last sum, we rewrite it as
\begin{align*}
 \frac{1}{N}\sum_{ p=1}^{ n-1}(n-p)\sum_{1\leq k,l\leq N}\left\{\mathbb{E}\textrm{tr}[ B_s^{ n-p}U_sSE_{k,l}SU_s^{\star} B_s^{p} U_sE_{l,k}U_s^{\star}] + \mathbb{E}\textrm{tr}[ B_s^{ n-p}U_sE_{k,l}U_s^{\star} B_s^{p} U_sSE_{l,k}SU_s^{\star}]  \right\}.
\end{align*}
Finally, the identity \eqref{magic} entails
\begin{align*}
   \frac{d}{ds} H_n(s)=&n\beta \mathbb{E}\textrm{tr}[ B_s^{ n-1}V_{2(t-s)}R]
+\frac{2}{N^2}\sum_{p=1}^{ n-1}(n-p)\mathbb{E}[\textrm{Tr}( B_s^{ n-p-1}V_{2(s-t)}R)\textrm{Tr}( B_s^{p-1}V_{2(t-s)} R)]
\\=&n\beta \mathbb{E}\textrm{tr}[ B_s^{ n-1}V_{2(t-s)}R]
+ n\sum_{p=1}^{ n-1}\mathbb{E}[\textrm{tr}( B_s^{ n-p-1}V_{2(t-s)}R)\textrm{tr}( B_s^{p-1}V_{2(t-s)} R)],
\end{align*}
so that formula \eqref{Exp} follows from the fundamental Theorem of calculus:
\begin{align*}
F_n(t)-\nu_n(2 t)=H_n(t)-H_n(0) = \int_0^t\frac{dH_n}{ds}(s)ds.
\end{align*}
\end{proof}
\begin{rem}
We can rewrite formula \eqref{Exp} as: 
\begin{align*}
F_n(t)=&\nu_n(2 t)+ n\beta \int_0^t \mathbb{E}\textrm{tr}[B_s^{ n-1}V_{2(t-s)}R]ds + n \sum_{p=1}^{n-1}\int_0^t \mathbb{E}\textrm{tr}( B_s^{ n-p-1}V_{2(t-s)}R)\mathbb{E}\textrm{tr}( B_s^{p-1}V_{2(t-s)} R) ds
\\&+n\sum_{p=1}^{n-1}\int_0^t {\rm cov}\left(\textrm{tr}( B_s^{ n-p-1}V_{2(t-s)}R), \textrm{tr}( B_s^{p-1}V_{2(t-s)} R)\right)ds. 
\end{align*}
Thus, the determination of the large-N limit of $F_n(t)$ will follow from the large-N limit of the covariance term. Once this limit will be determined, it will provide a general expression for the moments of the free Jacobi process. 
\end{rem}

\end{document}